\documentclass[12pt,a4paper]{amsart} 
\usepackage[dvips]{epsfig} 
\usepackage{amsmath,amssymb,euscript,graphicx,amsfonts,verbatim, bm}
\usepackage{enumerate,color}

\headsep=1cm \numberwithin{equation}{section}
\newtheorem{theorem}{Theorem}[section]
\newtheorem{corollary}[theorem]{Corollary}
\newtheorem{lemma}[theorem]{Lemma}
\newtheorem{example}[theorem]{Example}
\theoremstyle{definition}

\theoremstyle{Remark}
\newtheorem{remark}[theorem]{Remark}
\newtheorem{proposition}[theorem]{Proposition}

\numberwithin{equation}{section}
\usepackage[linktoc=all, hyperindex]{hyperref}%
\hypersetup{colorlinks, citecolor=green, filecolor=blue, linkcolor=red, urlcolor=black}

\newcommand{\charac}{\operatorname{char}}
\newcommand{\per}{\operatorname{per}}

\def\ZZ{{\mathbb Z}}
\def\CC{{\mathbb C}}

\def\FF{{\mathbb F}}

\begin{document}
\title[Vanishing Immanants ]{Vanishing Immanants}
\author[Cheraghpour \and Kuzma]{Hassan Cheraghpour$^{*}$ \and Bojan Kuzma}

\address{Hassan Cheraghpour:
University of Primorska, FAMNIT, Glagolja{\v s}ka 8, 6000 Koper, Slovenia.}
\email{cheraghpour.hassan@yahoo.com}

\address{Bojan Kuzma:
$ ^{1} $University of Primorska, FAMNIT, Glagolja{\v s}ka 8, 6000 Koper, Slovenia \and $ ^{2} $IMFM, Jadranska 19, 1000 Ljubljana, Slovenia.}
\email{bojan.kuzma@famnit.upr.si}

\thanks{This work is supported in part by the Slovenian Research Agency (research program P1-0285 and research projects N1-0210, N1-0296 and J1-50000).}
\thanks{$^{*}$Corresponding author: Hassan Cheraghpour}
\maketitle

\begin{abstract}
\noindent
We classify all the irreducible characters of a symmetric group such that the induced immanant function $ d_{\chi} $ vanishes identically on alternate matrices with the entries in the complex field.

\end{abstract}

\maketitle

\section{\textbf{Introduction}} 
\noindent
Let $ \CC $ be the field of the complex numbers, $ S_{n} $ be the permutation group acting on the set $ \{ 1,\dots, n\} $, let $ \mathbb{M}_{n}(\CC) $ denote the space of all $ n$-by-$n $ matrices with entries in $ \CC $, and
$$ \mathbb{A}_{n}(\CC)=\{A =(a_{ij}) \in \mathbb{M}_{n}(\CC) \colon a_{ij}=-a_{ji} \text{ if } i<j,~a_{ii}=0 \} $$
be the subspace of the alternate matrices over $ \CC $.
The \textit{immanant} associated with the irreducible character $ \chi $ of $ S_{n} $ is the function $ d_{\chi} \colon \mathbb{M}_{n}(\CC) \rightarrow \CC $ given by 
$$ d_{\chi}(A) =\sum_{\sigma \in S_{n}}\chi(\sigma) \prod_{i=1}^{n} a_{i\sigma(i)}. $$

Notice that immanant generalizes the \textit{permanent},
$ \per(A)=\sum_{\sigma \in S_{n}}\prod_{i=1}^{n} a_{i\sigma(i)} $, where $ \chi=1 $ is the principal character, 
and the \textit{determinant},
$ \det(A)=\sum_{\sigma \in S_{n}} \varepsilon(\sigma)\prod_{i=1}^{n} a_{i\sigma(i)}$, where $ \chi=\varepsilon $ is the alternating character of $ S_{n} $.

Since computing an immanant $ d_{\chi} $ may not be always easy, one may hope to find a suitable linear map $ T $ that converts $ d_{\chi} $ into an immanant $ d_{\chi'} $.
Given a subspace $ V \subseteq \mathbb{M}_{n}(\CC) $, we say that a linear transformation $ T \colon V \rightarrow V $ \textit{converts} an immanant $ d_{\chi} $ into an immanant $ d_{\chi'} $ on $ V $ if $ d_{\chi'}(T(X))=d_{\chi}(X) $ for all $ X \in V $. In the case where $ \chi=\chi' $, one says that these maps \textit{preserve} the immanant $ d_{\chi} $.

The maps which convert/preserve immanants have been investigated in many papers. 
First, in 1994, Duffner \cite{Duf94} characterized  linear maps which preserve the immanant $ d_{\chi} $ on $ \mathbb{M}_{n}(\CC) $, $ n \geq 3 $, where $ \chi \notin \{ 1 , \varepsilon \} $ and proved that linear preservers of immanants are always non-singular. The main result shows that, for $n\ge 4$ and $ \chi \notin \{ 1 , \varepsilon \} $, the group of such  linear preservers is generated by  permutation of rows, permutation of columns, transposition, and a Hadamard multiplication with  a matrix that has some additional properties.
In 1997, Coelho and Duffner \cite{Coe97}, showed that the same result is true for linear immanant preservers on $ \mathbb{S}_{n}(\CC) $, the space of symmetric $ n $-by-$ n $ matrices where again $ \chi \notin \{ 1 , \varepsilon \} $. The only difference is that permutation of rows and columns should be done in unison.
In 2012, they \cite{Coe12} extended this result to the alternate matrices and classified the linear maps which preserve the immanant $ d_{\chi} $ on $ \mathbb{A}_{n}(\CC) $ (either $ n \geq 6 $, or $ n=4 $ and $ \chi=(2,2) $), where $ \chi \notin \{ 1 , \varepsilon, (n-1, n), (2,1^{n-2}) \} $. The result is exactly the same as for linear immanant preservers on $ \mathbb{S}_{n}(\CC) $, except for immanants which vanish identically on $ \mathbb{A}_{n}(\CC) $ (clearly any linear map will preserve such immanants).

Next, in 1998, Coelho and Duffner \cite{Coe98} proved that there is no linear map $ T\colon \mathbb{M}_{n}(\CC) \rightarrow \mathbb{M}_{n}(\CC) $, $ n \geq 3 $ that converts an immanant $ d_{\chi} $ into an immanant $ d_{\chi'} $, where $ \chi \neq \chi' $.
In 2003, they \cite{Coe03} extended this result to the symmetric matrices. They proved that there is no linear map $ T\colon \mathbb{S}_{n}(\CC) \rightarrow \mathbb{S}_{n}(\CC) $, $ n \geq 3 $ that converts an immanant $ d_{\chi} $ into an immanant $ d_{\chi'} $, where $ \chi \neq \chi' $.
In 2021, Duffner et al. \cite{Duf21.1} extended this result to the alternate matrices. They proved that there exists no map $ T\colon \mathbb{A}_{n}(\CC) \rightarrow \mathbb{A}_{n}(\CC) $, $ n \geq 6 $, satisfying $ d_{\chi}(A+\alpha B)=d_{\chi'}(T(A)+\alpha T(B)) $, $ \alpha \in \CC $, where $ \chi $ and $ \chi' $ are not proportional on the subset $ P_{n} $ of all permutations of $ S_{n} $ with no cycles of odd length in the decomposition into the product of disjoint cycles (i.e., there is no $ 0\neq \beta \in \CC $ such that $ \chi'|_{P_{n}}=\beta . \chi|_{P_{n}}$). Also, they characterized such maps $ T $ if $ \chi $ and $ \chi' $ are proportional on $ P_{n} $, but the induced immanants do not vanish identically on $ \mathbb{A}_{n}(\CC) $, and proved that $ T $ is bijective and linear.
Then, they extended \cite{Duf21.2} this result to the alternate matrices of order $ 4 $ and proved that there are no linear maps $ T\colon \mathbb{A}_{4}(\CC) \rightarrow \mathbb{A}_{4}(\CC) $ that converts an immanant $ d_{\chi} $ into an immanant $ d_{\chi'} $ for all matrices $ A \in \mathbb{A}_{4}(\CC) $, where $ \chi, \chi' \in \{1, \varepsilon, (2,2) \} $ are two distinct irreducible characters of $ S_{4} $.

In 2006, Coelho and Duffner \cite{Coe06} proved that if $ \chi $ and $ \chi' $ are arbitrary irreducible complex characters of $ S_{n} $ and $ T\colon \mathbb{M}_{n}(\CC) \rightarrow \mathbb{M}_{n}(\CC) $, $ n \geq 3 $, is a surjective map satisfying the condition $ d_{\chi}(A+\alpha B)=d_{\chi'}(T(A)+\alpha T(B)) $, for all $ A, B \in \mathbb{M}_{n}(\CC) $ and all $ \alpha \in \CC $, then $ T $ is linear. Their main theorem combined with known results on linear preservers/converters of immanants, allows us to characterize the pairs $ (\chi, \chi' ) $ for which such maps exist and, in the cases they exist, to obtain the respective description.

In 2008, Kuzma \cite{Kuz}, following the idea of \cite{Tan}, showed that if   $ \FF $ is any field with at least $n+1$ elements ($n\ge 3$) and $\chi,\chi'$ are irreducible characters of $S_n$, then the immanant converter on matrix pencils $ T \colon \mathbb{M}_{n}(\FF) \rightarrow \mathbb{M}_{n}(\FF) $,  $ d_{\chi}(A+\alpha B)=d_{\chi'}(T(A)+\alpha T(B)) $, $ A, B \in \mathbb{M}_{n}(\FF) $  is automatically  linear and bijective. His arguments  rely on the property that  for every pair of integers $(i,j)\in\{1,\dots,n\}$ there exists a permutation $\sigma\in S_n$ with $\sigma(i)=j$ and $\chi(\sigma)\ne 0\in \FF$. Note that this is automatically satisfied in fields of characteristic $0$. However if characteristic of the field is nonzero, it is not obvious that such a permutation $\sigma$ exists. We postpone this question and also give the correction of the main result of \cite{Kuz} in Addendum (section \ref{Addendum}).
This was extended, in 2014, by Coelho et al. \cite{Coe14} to matrix pencils consisting of symmetric matrices. Namely, following \cite{Coe06, Kuz}, they proved that the maps $ T\colon \mathbb{S}_{n}(\CC) \rightarrow \mathbb{S}_{n}(\CC) $ satisfying the condition $ d_{\chi}(A+\alpha B)=d_{\chi'}(T(A)+\alpha T(B)) $, for any fixed irreducible characters $ \chi, \chi'\colon S_{n} \rightarrow \CC $, any $ A, B \in \mathbb{S}_{n}(\CC) $, and any $ \alpha \in \CC $, are automatically linear and bijective.
In 2017, Duffner and Guterman \cite{Duf17} extended this result and studied converters between $ d_{\chi} $ and $ d_{\varepsilon}=\det $, where unlike \cite{Coe14} they considered matrix pencils spanned by singular symmetric matrices.

As already remarked, some immanants vanish identically on $\mathbb{A}_n(\CC)$. 
In \cite[Corollary 2.2]{Coe12}, it was shown that for every even integer $ n $, the immanant induced by a triangular character vanishes identically on $\mathbb{A}_n(\CC)$. In fact, the triangular character is the only character which vanishes on all the conjugacy classes whose cycle decomposition contains at least one transposition (see \cite{Coe97.2}). 




The main goal in this paper is to find all the irreducible characters such that the induced immanant function $ d_{\chi} $ vanishes identically on $ \mathbb{A}_{n}(\CC) $. 
Obviously, for $ A \in \mathbb{M}_{n}({\CC}) $ we have $ d_{\chi}(A^{T})=d_{\chi}(A) $, so for alternate matrices, $ d_{\chi}(A)=d_{\chi}(A^{T})=d_{\chi}(-A)=(-1)^{n}d_{\chi}(A) $. Hence, if $ n $ is odd, then $ d_{\chi}(A)=0 $ for any $ A \in \mathbb{A}_{n}(\CC) $. So, for finding the immanants which vanish identically, it is enough to consider the irreducible characters of $ S_{n} $ in the case when $ n $ is even. In this case, Duffner et al. \cite[Proposition 2.11]{Duf21.1} proved the following result.
\begin{proposition}\label{impn}
Let $ A \in \mathbb{A}_{n}(\CC) $. Then $ d_{\chi}(A) =\sum_{\sigma \in P_{n}}\chi(\sigma) \prod_{i=1}^{n} a_{i\sigma(i)} $.
\end{proposition}
So, for finding the vanishing immanants, it is enough to find the irreducible characters which are identically zero for any permutation $ \sigma \in P_{n} $.

Recall that irreducible characters of $ S_{n} $ are in bijective correspondence with partitions of integer $ n $ (see \cite[Theorem 2.4.6]{Sa}). Also, partitions are usually visualized by their Young diagrams.
With this in mind, our first main result is contained in Theorems \ref{Main1} and \ref{character}, where by using the recursive Murnaghan-Nakayama rule, we find all the Young diagrams with induced characters vanishing identically on $ P_{n} $. The second main result is Corollary \ref{m.cor} which classifies the immanants vanishing identically on alternate matrices.

\section{\textbf{Preliminaries}}
\noindent
A \textit{decomposition} $ \lambda $ of an integer $ n \geq 1 $ is a sequence $ (\lambda_{1},\dots, \lambda_{t}) $ where $ \lambda_{i} \in \mathbb{N} $ and $ \sum_{i=1}^{t}\lambda _{i}=n$. 
A \textit{partition} $ \lambda $ of an integer $ n $ is a decomposition $ (\lambda_{1},\dots, \lambda_{t}) $ of $ n $ where $ \lambda_{1} \geq \dots \geq \lambda_{t} $. If $ \lambda_{r}=\dots=\lambda_{s} $ for some $ r, s \in \{1, \dots, t \} $, then $ (\lambda_{1}, \dots ,\lambda_{r},\dots,\lambda_{s},\dots, \lambda_{t}) $ will be written as  $ (\lambda_{1},\dots, \lambda_{r}^{s-r+1}, \dots, \lambda_{t}) $ for simplicity.
A \textit{shape} $ {\bm \lambda} $ is a geometric object which consists finitely many one by one square cells arranged in left-justified rows and top-justified columns where adjacent cells are edge-connected. Shapes are in bijective correspondence with decompositions of integer $ n $, i.e., the number of cells in a shape. This correspondence is given by
$$  {\bm \lambda} \leftrightarrow (\lambda_{1},\dots, \lambda_{t}), $$
where $ \lambda_{i} $ is the number of cells in a row $ i $ and $ t $ is the number of rows of $ {\bm \lambda} $. We also use the notation 
$$ | {\bm \lambda} | =\sum_{i=1}^{t}\lambda _{i}=n. $$

If the number of cells in consecutive rows of a shape is decreasing, we call it a \textit{valid shape} or a \textit{\emph{(}Young\emph{)} diagram}. These are in bijective correspondence with partitions of integer $ n $.
Notice that a Young diagram $ {\bm \lambda}=(\lambda_{1},\dots, \lambda_{t}) $ is a shape such that $ \lambda_{1} \geq \dots \geq \lambda_{t} $.

A \textit{rim-hook} $ {\bm {\bm \zeta}} $ of a Young diagram $ {\bm \lambda} $ consists
of a chain of consecutive edge-connected cells by taking all the cells always leftwards or downwards such that the obtained shape $ {\bm \lambda} \setminus {\bm {\bm \zeta}} $ is also a Young diagram.
If all the cells of $ {\bm \zeta} $ are only in one row (respectively, one column), we call it a \textit{horizontal rim-hook} (respectively, \textit{vertical rim-hook}).

From now on, we simply say \textit{diagram} instead of Young diagram.

\section{\textbf{Destructibility of diagrams}}
\noindent
Let $ {\bm \lambda} $ be a diagram. A \textit{domino} $ D $ in $ {\bm \lambda} $ is a pair of vertically or horizontally connected cells of $ {\bm \lambda} $ that intersect in a common edge, i.e., form a $ 2 \times 1 $ or $ 1 \times 2 $ rectangle. We call these two cells a \textit{vertical domino} or a \textit{horizontal domino}, respectively.

A \textit{domino rim-hook} is a rim-hook of length 2. Notice that it coincides with a single vertical or horizontal domino, respectively.

The diagram $ {\bm \lambda} $ is called \textit{destructible} if there exists a recursive procedure $ R_{{\bm \lambda}} $ which removes all the cells from $ {\bm \lambda} $ such that at each step we remove a single domino rim-hook (where a rim-hook refers to a diagram from this particular step, and so it is different at each step). If there is no such recursive procedure, we call the diagram \textit{indestructible}. For example, the diagram $ {\bm \lambda}=(2n) $ is a destructible diagram, since it consists of $ n $ horizontal dominoes. Examples of indestructible diagrams will be provided in the sequel.

Two dominoes are \textit{disjoint} if they do not share the same cell. It is easy to see that disjoint domino rim-hooks will never share the same edge, however they can still intersect in a common vertex. This happens for example in a diagram $ (3,3,2) $.

\begin{lemma}\label{D1-D2}
Let $ {\bm \lambda} $ be a diagram with two disjoint domino rim-hooks $ D $ and $ E $. Then $ E $ is a domino rim-hook of $ {\bm \lambda} \setminus D $.
\end{lemma}

\begin{proof}
Write $ {\bm \lambda}=(\lambda_{1}, \dots, \lambda_{t}) $, where $ \lambda_{1} \geq \dots \geq \lambda_{t} $. 
First, if dominoes $ D $ and $ E $ are both horizontal, let $  i < j $ be the corresponding rows which contain them. Since $  {\bm \lambda} \setminus D $ and $ {\bm \lambda} \setminus E $ are both valid shapes, then $ \lambda_{i}-2 \geq \lambda_{i+1} $ and $ \lambda_{j}-2 \geq \lambda_{j+1} $. Therefore
$$ ({\bm \lambda} \setminus D) \setminus E =(\lambda_{1},\dots,\lambda_{i-1},\lambda_{i}-2,\lambda_{i+1},\dots,\lambda_{j-1}, \lambda_{j}-2,\lambda_{j+1},\dots,\lambda_{t}) $$
is a valid shape, and hence $ E $ is a domino rim-hook of $ {\bm \lambda} \setminus D $.

Second, if one of $ D,E $ is vertical, in rows $ i-1 $ and $ i $, and another is horizontal, in row $ j > i$, then, by a similar argument, $\lambda_{i-1}-1 =\lambda_{i}-1 \geq \lambda_{i+1} $ and $ \lambda_{j}-2 \geq \lambda_{j+1} $. Therefore 
$$ ({\bm \lambda} \setminus D) \setminus E =(\lambda_{1},\dots, \lambda_{i-2}, \ \lambda_{i-1}-1,\lambda_{i}-1,\ \lambda_{i+1},\dots,\lambda_{j-1}, \ \lambda_{j}-2,\ \lambda_{j+1},\dots,\lambda_{t}) $$
is again a valid shape, and the same conclusion holds. Likewise, we argue if $ j < i-1 $ or if $ D,E $ are both vertical (and $\lambda_{i-1}-1 =\lambda_{i}-1 \geq \lambda_{i+1} $ and $\lambda_{j-1}-1 =\lambda_{j}-1 \geq \lambda_{j+1} $).
\end{proof}

The following lemma is crucial for our subsequent investigation.

\begin{lemma}\label{dest. domino}
Let $ {\bm \lambda} $ be a diagram with $| {\bm \lambda} |=2n$, and let $D$ be one of its domino rim-hooks. Then $ {\bm \lambda} $ is destructible if and only if $ {\bm \lambda} \setminus D $ is destructible.
\end{lemma}
\begin{proof}
Assume $ {\bm \lambda} \setminus D $ is destructible and let 
\begin{equation*}
R_{{\bm \lambda} \setminus D}=\{ {\bm \lambda} \setminus D \rightarrow ({\bm \lambda} \setminus D)^{(1)} \rightarrow \dots \rightarrow ({\bm \lambda} \setminus D)^{(n-2)} \rightarrow ({\bm \lambda} \setminus D)^{(n-1)}=(0) \}
\end{equation*}
be a recursive procedure for $ {\bm \lambda} \setminus D $. Then
\begin{equation}\label{R}
R_{{\bm \lambda}}=\{ {\bm \lambda} \rightarrow {\bm \lambda} \setminus D \rightarrow ({\bm \lambda} \setminus D)^{(1)} \rightarrow \dots \rightarrow ({\bm \lambda} \setminus D)^{(n-2)} \rightarrow ({\bm \lambda} \setminus D)^{(n-1)}=(0) \}
\end{equation}
is a recursive procedure for $ {\bm \lambda} $. Therefore $ {\bm \lambda} $ is destructible.

Conversely, assume $ {\bm \lambda} $ is destructible. We use induction on the number $ n $ of dominoes in $ {\bm \lambda} $. The case $ n=1 $ is trivial. If $ n=2 $, then $ {\bm \lambda} $ can take one among the five different forms $ (4) $, $ (3,1) $, $ (2,2) $, $ (2,1,1) $ and $ (1,1,1,1) $. By removing a single domino rim-hook from any one of these, we get either a diagram $ (2) $ or a diagram $ (1,1)$, both of which are obviously destructible. This proves the base of induction.

Assume now that for any destructible diagram $ {\bm \gamma} $ with $ |{\bm \gamma}|\leq 2n $, the diagram $ {\bm \gamma} \setminus D $, where $ D $ is any of its domino rim-hooks, is always destructible. Now, let
\begin{equation*}
{\bm \lambda}=(\lambda_{1}, \dots, \lambda_{t}, 0, 0, \dots), \;\;\;\; |{\bm \lambda}|=2n+2
\end{equation*}
be destructible with $ n+1 $ dominoes, let $ D $ be its domino rim-hook, and let
\begin{equation*}
R_{{\bm \lambda}}=\{ {\bm \lambda} \rightarrow {\bm \lambda}^{(1)} \rightarrow \dots \rightarrow {\bm \lambda}^{(n)} \rightarrow {\bm \lambda}^{(n+1)}=(0) \}
\end{equation*}
be a recursive procedure for $ {\bm \lambda} $. Therefore, $ {\bm \lambda}^{(1)} $ is also destructible. Notice that $ {\bm \lambda}^{(1)}={\bm \lambda} \setminus E $, where $ E $ is a domino rim-hook of $ {\bm \lambda} $. We have four cases:\\

1. Assume that at the first step domino $ D $ is completely removed. Then $ D=E $, so $ {\bm \lambda} \setminus D={\bm \lambda}^{(1)}$ is destructible.

2. Assume that at the first step none of the cells of domino $ D $ is removed. By Lemma \ref{D1-D2}, $ D $ is also a domino rim-hook of $ {\bm \lambda}^{(1)} $. Since $ {\bm \lambda}^{(1)} $ is destructible with $ |{\bm \lambda}^{(1)}|=2n $, by the induction, $ {\bm \lambda}^{(1)} \setminus D=({\bm \lambda} \setminus E) \setminus D=({\bm \lambda} \setminus D) \setminus E $ is also destructible. Therefore, there exists a recursive procedure
$$ \{ ({\bm \lambda} \setminus D) \rightarrow ({\bm \lambda} \setminus D)^{(1)}=({\bm \lambda} \setminus D) \setminus E \rightarrow \dots \rightarrow (0) \} $$
for $ {\bm \lambda} \setminus D $.

3. Assume that $ D $ is a horizontal domino rim-hook and at the first step, domino $ E $ removed a single cell from  $ D $. Then $ E $ was vertical. Let $ D $ belong to a row $ i $. Then $ \lambda_{i} \geq \lambda_{i+1} + 2 $, and hence $ E $ cannot belong to rows $ i $ and $ i+1 $. Therefore, since $ {\bm \lambda} \setminus E $ is a valid shape,
\begin{equation}\label{D1-E1}
\lambda_{i-1}=\lambda_{i} \geq \lambda_{i+1} + 2. 
\end{equation}
Let $ E' $ be the vertically placed domino in the last cells of rows $ i-1 $ and $ i $ of $ {\bm \lambda}^{(1)}= {\bm \lambda} \setminus E $. By \eqref{D1-E1},
\begin{equation}\label{D2-E2}
{\bm \lambda}^{(1)} \setminus E' =({\bm \lambda} \setminus E) \setminus E'=(\lambda_{1}, \dots, \ \lambda_{i-1}-2,\lambda_{i}-2, \ \lambda_{i+1}, \dots, \lambda_{t},0, 0, \dots)
\end{equation}
is a valid shape, and hence $ E' $ is a domino rim-hook of a destructible diagram $ {\bm \lambda}^{(1)} $ with $ n $ dominoes. 
Notice that \eqref{D2-E2} can also be written as $ ({\bm \lambda} \setminus D) \setminus D' $, where $ D' $ is horizontal domino occupying the last two cells of row $ i-1 $, and since \eqref{D2-E2} is a valid shape, $ D' $ is a domino rim-hook of $ {\bm \lambda} \setminus D $.

By the induction, $ {\bm \lambda}^{(1)} \setminus E'=({\bm \lambda} \setminus E) \setminus E'=({\bm \lambda} \setminus D) \setminus D' $ is also destructible, and hence there exists a recursive procedure 
$$ \{ ({\bm \lambda} \setminus D) \rightarrow ({\bm \lambda} \setminus D)^{(1)}=({\bm \lambda} \setminus D) \setminus D' \rightarrow \dots \rightarrow (0) \} $$
for $ {\bm \lambda} \setminus D $.

4. Assume that $ D $ is a vertical domino rim-hook and at the first step, domino $ E $ removed a single cell from  $ D $. Then $ E $ was horizontal. Let $ D $ belong to rows $ i -1 $ and $ i $. Then $ \lambda_{i-1}=\lambda_{i} \geq \lambda_{i+1} + 1 $, and hence $ E $ cannot belong to row $ i-1 $. Therefore, since $ {\bm \lambda} \setminus E $ is a valid shape,
\begin{equation}\label{D3-E3}
\lambda_{i-1}=\lambda_{i} \geq \lambda_{i+1} + 2. 
\end{equation}
Let $ E' $ be the horizontally placed domino in the last two cells of row $ i-1 $ of $ {\bm \lambda}^{(1)}= {\bm \lambda} \setminus E $. By \eqref{D3-E3},
\begin{equation}\label{D4-E4}
{\bm \lambda}^{(1)} \setminus E' =({\bm \lambda} \setminus E) \setminus E'=(\lambda_{1}, \dots, \lambda_{i-2}, \ \lambda_{i-1}-2,\lambda_{i}-2, \ \lambda_{i+1}, \dots, \lambda_{t},0, 0, \dots)
\end{equation}
is a valid shape, and hence $ E' $ is a domino rim-hook of a destructible diagram $ {\bm \lambda}^{(1)} $ with $ n $ dominoes. 
Notice that \eqref{D4-E4} can also be written as $ ({\bm \lambda} \setminus D) \setminus D' $, where $ D' $ is vertical domino occupying the last cells of rows $ i-1 $ and $ i $, and since \eqref{D4-E4} is a valid shape, $ D' $ is a domino rim-hook of $ {\bm \lambda} \setminus D $.

By the induction, $ {\bm \lambda}^{(1)} \setminus E'=({\bm \lambda} \setminus E) \setminus E'=({\bm \lambda} \setminus D) \setminus D' $ is also destructible, and hence there exists a recursive procedure 
$$ \{ ({\bm \lambda} \setminus D) \rightarrow ({\bm \lambda} \setminus D)^{(1)}=({\bm \lambda} \setminus D) \setminus D' \rightarrow \dots \rightarrow (0) \} $$
for $ {\bm \lambda} \setminus D $.
\end{proof}

To simplify, let us denote the diagram $ {\bm \lambda}=(m, m-1, \dots, i) $, $ m \geq 2 $, by $ \bigtriangledown_{i} ^{m}$. If $ i=1 $ the diagram is called \textit{triangular}.

\begin{lemma}\label{Tr2}
A diagram $ {\bm \lambda} $ is triangular if and only if there is no domino rim-hook in $ {\bm \lambda} $.
\end{lemma}
\begin{proof}
A vertical domino rim-hook requires two consecutive rows of equal length, while a horizontal domino rim-hook requires that there exist two consecutive rows whose lengths differ by at least two, but neither is possible if $ {\bm \lambda}=\bigtriangledown_{1}^{m} $.

Now, assume that $ {\bm \lambda} $ is not triangular. Then there exist two consecutive rows $ i $ and $ i+1 $ such that $ d_{i}=\lambda_{i}-\lambda_{i+1} \neq 1 $. If $ d_{i} \geq 2 $, then there exists a horizontal domino rim-hook in row $ i $. If $ d_{i}=0 $, let $ k \geq i+1 $ be the last row with $ \lambda_{i}=\lambda_{i+1}= \dots=\lambda_{m} $. Then there exists a vertical domino rim-hook in rows $ m-1 $ and $ m $. 
\end{proof}

Clearly if there is no recursive procedure, the diagram must be indestructible. Hence, by Lemma \ref{Tr2} we can state the following corollary:
\begin{corollary}\label{Tr3}
Every triangular diagram $ \bigtriangledown_{1}^{m} $, where $ m > 1 $, is indestructible.
\end{corollary}

We generalize the above corollary in the following lemma that classifies the indestructible diagrams.

\begin{lemma}\label{Triangle-Remove}
Let $ {\bm \lambda} $ be a diagram with $ |{\bm \lambda}| $ even. The followings are equivalent:
\begin{itemize}
\item[(i)] $ {\bm \lambda} $ is indestructible.
\item[(ii)] $ {\bm \lambda} $ is triangular or else there exists a recursive procedure of domino rim-hook removal which transforms $ {\bm \lambda} $ into a triangular diagram.
\item[(iii)] Every recursive procedure of domino rim-hook removal applied on $ {\bm \lambda} $ will eventually end up in a triangular diagram.

\item[(iv)] There exists a triangular diagram $ \bigtriangledown_{1}^{m} $ such that every recursive procedure of domino rim-hook removal will transform $ {\bm \lambda} $ into $ \bigtriangledown_{1}^{m} $.
\end{itemize}

\end{lemma}

\begin{proof}
(i) $ \Rightarrow $ (iii). Suppose that there exists a recursive procedure of domino rim-hook removal applied on $ {\bm \lambda} $ which does not end up in a triangular diagram. Then, by Lemma \ref{Tr2}, at each step there exists a domino rim-hook until removing all the cells from $ {\bm \lambda} $. Therefore $ {\bm \lambda} $ is destructible, contradicting (i).

(iii) $ \Rightarrow $ (ii) is obvious.

(ii) $ \Rightarrow $ (i). If $ {\bm \lambda} $ is triangular, then by Corollary \ref{Tr3}, $ {\bm \lambda} $ is indestructible. Now, assume that there exists a recursive procedure of domino rim-hook removal which transforms $ {\bm \lambda} $ into a triangular diagram. Since, by Lemma \ref{Tr2}, the transformed triangular diagram is indestructible, by recursively using Lemma \ref{dest. domino}, $ {\bm \lambda} $ is also indestructible.

(iv) $ \Rightarrow $ (iii) is obvious.

(iii) $ \Rightarrow $ (iv). We recall that $ {\bm \lambda} $ consists of cells indexed by left-justified rows and top-justified columns where adjacent cells are edge-connected. We denote the cells of row $ i $ in the diagram $ {\bm \lambda} $ by $ C_{i1}, C_{i2},\dots, C_{i\lambda_{i}} $ here $ C_{ij} $ belongs to row $ i $ and column $ j $ and the top-most row and the left-most column have index~$ 1 $. With this notation each vertical/horizontal domino can be written as a set
\begin{equation}\label{DvDh}
D^{v}=\{ C_{(i-1)j}, C_{ij} \}, \quad \hbox{respectively} \quad D^{h}=\{ C_{i(j-1)}, C_{ij} \}
\end{equation}
for some $ i $ and $ j $.

Let $ m , m' > 1 $, where $ m < m' $, be two distinct integers such that two recursive procedures of domino rim-hook removal $ R_{{\bm \lambda}} $ and $ R'_{{\bm \lambda}} $ applied on $ {\bm \lambda} $ will eventually end up in triangular diagrams $ \bigtriangledown_{1} ^{m} $ and $ \bigtriangledown_{1} ^{m'} $, respectively. By reversing the procedure $ R_{{\bm \lambda}} $, the diagram $ {\bm \lambda} $ can be obtained from $ \bigtriangledown_{1} ^{m} $ by adding domino rim-hooks. We then apply the procedure $ R'_{{\bm \lambda}} $ to see that the diagram $ \bigtriangledown_{1} ^{m'} $ can also be obtained from $ \bigtriangledown_{1} ^{m} $ by adding domino rim-hooks. Then $ \bigtriangledown_{1} ^{m'} \setminus \bigtriangledown_{1} ^{m} $ must be covered with non-overlapping, i.e., disjoint dominoes. We will show that this is impossible. Notice that $ C_{1m'} $, the last cell of first row in $ \bigtriangledown_{1} ^{m'} \setminus \bigtriangledown_{1} ^{m} $, is also the only cell in its last column, and therefore it can only be covered with a horizontal domino $ D_{1}^{h}=\{ C_{1(m'-1)}, C_{1m'} \} $. Then $ C_{2(m'-1)} $, the last cell of second row, cannot be covered with a vertical domino $ D_{1}^{v}=\{ C_{1(m'-1)}, C_{2(m'-1)} \} $, because $ D_{1}^{v} $ intersects $ D_{1}^{h} $ in the cell $ C_{1(m'-1)} $. Therefore $ C_{2(m'-1)} $ is also covered with a horizontal domino $ D_{2}^{h}=\{ C_{2(m'-2)}, C_{2(m'-1)} \} $. This allows us to recursively proceed forwards and show that the last cell of each succeeding row can only be covered with a horizontal domino. This is a contradiction, because the last row in $ \bigtriangledown_{1} ^{m'} \setminus \bigtriangledown_{1} ^{m} $ contains only one cell and it cannot be covered with a horizontal domino.
\end{proof}

Now, we give the prototypical example of indestructible diagrams, which generalizes triangular ones from Corollary \ref{Tr3}.

\begin{example}\label{Initial pattern}
\textnormal{
The diagram $ {\bm \lambda}=(2n-\frac{m(m+1)}{2},\bigtriangledown_{1} ^{m}) $, where $ m \geq 2 $, is indestructible.
To see this we first rewrite $ {\bm \lambda}=(m+k,\bigtriangledown_{1} ^{m}) $. By removing recursively domino rim-hooks from the first row, we have two cases.
If $ k $ is odd, then the obtained diagram $ (m+1,\bigtriangledown_{1} ^{m})=\bigtriangledown_{1} ^{m+1} $ is triangular, hence, by Lemma \ref{Triangle-Remove}, is indestructible.
If $ k $ is even, then the diagram $ (m,\bigtriangledown_{1} ^{m}) $ is obtained. Now, by recursively removing the vertical domino rim-hooks from two consecutive rows $ \{i ,i+1 \} $, $ i \in \{1, \dots, m-1 \}$, we eventually get the triangular diagram $ \bigtriangledown_{1}^{m-1} $. Then, by Lemma \ref{Triangle-Remove}, $ {\bm \lambda} $ is again indestructible.
}
\end{example}

\begin{remark}\label{reversing}
\textnormal{
By reversing the procedure described in Lemma \ref{Triangle-Remove}, we obtain one possibility to construct all indestructible diagrams. We just start with a triangular diagram and recursively add dominoes in such a way that at each step, the newly added domino becomes a domino rim-hook for the obtained diagram. Notice that by applying this procedure to a single vertical or horizontal domino as a starting point, one can also build all possible destructible diagrams.
}
\end{remark}
Another possibility to obtain all (in)destructible diagrams will be given in the next section.

\section{\textbf{Two rules which preserve the (in)destructibility of diagrams}}
\noindent
In this section, we state and prove two main rules about destruction diagrams.
If $ D $ is a horizontal (respectively, vertical) domino, then we define its \textit{transpose} to be a vertical (respectively, horizontal) domino.
\begin{corollary}[First rule]\label{Rule 1}
Suppose a diagram $ {\bm \lambda}' $ is obtained from diagram $ {\bm \lambda} $ by removing a domino rim-hook $ D $ and placing it \emph{(}or its transpose\emph{)} to any other row, including the new one after the last row, or to any two consecutive rows with the same lengths, including two new ones after the last row, of $ {\bm \lambda} $. Then $ {\bm \lambda} $ is destructible if and only if $ {\bm \lambda}' $ is destructible.
\end{corollary}
\begin{proof}
Denote the transferred domino by $ D' \subseteq {\bm \lambda}' $. Notice that $ {\bm \lambda} \setminus D={\bm \lambda}' \setminus D' $. Hence $ D' $ is a domino rim-hook of $ {\bm \lambda}' $ and Lemma \ref{dest. domino} finishes the proof.
\end{proof}

\begin{example}
\textnormal{
Consider a horizontal domino rim-hook $ D $ in the first row of $ {\bm \lambda}=(7,5,3,3,2) $. Applying the first rule can change $ {\bm \lambda} $ into one of the following diagrams
\begin{align*}
{\bm \lambda}^{'}&=(6,6,3,3,2), &{\bm \lambda}^{'}&=(5,5,5,3,2),  &{\bm \lambda}^{'}&=(5,5,4,4,2),\\
 {\bm \lambda}^{'}&=(5,5,3,3,2,2), & {\bm \lambda}^{'}&=(5,5,3,3,2,1,1).
\end{align*}
}
\end{example}

For any diagram $ {\bm \lambda} $, its \textit{transpose}, $ {\bm \lambda}^{T}=(\gamma_{1}, \dots, \gamma_{s}) $ is a diagram where $ \gamma_{i} $ is the length of the $ i $'th column of $ {\bm \lambda} $, that is, $ {\bm \lambda}^{T} $ is a diagram obtained from $ {\bm \lambda} $ by interchanging the rows and columns of $ {\bm \lambda} $. It is easy to see that $ \gamma_{i}=|\{ j \colon \lambda_{j} \geq i \}| $.
 Therefore the partition determined by $ {\bm \lambda}^{T} $ is the conjugate partition of the one determined by $ {\bm \lambda} $.
Now, we have the second rule:
\begin{lemma}[Second rule]\label{Rule 2}
A diagram $ {\bm \lambda} $ is destructible if and only if its transpose, $ {\bm \lambda}^{T} $, is destructible.
\end{lemma}

\begin{proof}[Sketch of the proof]
A horizontal (respectively, vertical) domino of $ {\bm \lambda} $ is a vertical (respectively, horizontal) domino of $ {\bm \lambda}^{T} $.
Moreover, for any rim-hook $ D $ of $ {\bm \lambda} $, we have $({\bm \lambda} \setminus D)^T= ({\bm \lambda} \setminus D)^T$ and $ D^{T} $, the transpose of $ D $, is a domino rim-hook of $ {\bm \lambda}^{T} $.
\end{proof}


Now, as we mentioned in the previous section, using the above two rules, other possibilities are given to obtain all (in)destructible diagrams.
\begin{lemma}\label{single row}
Let $ {\bm \lambda} $ be a diagram with $ | {\bm \lambda} |=2n$. Then $ {\bm \lambda} $ is destructible if and only if the diagrams $ {\bm \lambda} $ and $ (2n) $ can be transformed to each other by applying the above two rules. Moreover, $ {\bm \lambda} $ is indestructible if and only if there exists an integer $ m > 1 $ such that the diagrams $ {\bm \lambda} $ and \mbox{$ (2n-\frac{m(m+1)}{2},\bigtriangledown_{1} ^{m}) $} can be transformed to each other by applying the above two rules.
\end{lemma}
\begin{proof}
Using Corollary \ref{Rule 1} we can transfer recursively horizontal/vertical domino rim-hooks that are in rows $ 2, 3, \dots $ and place them horizontally to the first row. If all the dominoes from other rows can be transferred to the first row, then the obtained diagram is $ (2n) $. But, if after some steps there is no way to transfer domino rim-hooks from other rows to the first row, then there is no way to remove domino rim-hooks from those rows. In this case, by Lemma \ref{Tr2} there exists an integer $ m $ such that the obtained diagram is of the form $ (2n-\frac{m(m+1)}{2},\bigtriangledown_{1} ^{m}) $.

If $ m=1 $, the obtained diagram is $ (2n-1,1) $, so that by using the second rule, we transpose into $ (2n-1,1)^{T}=(2,1^{2n-2}) $. Then using Corollary \ref{Rule 1}, after $ n-1 $ times removing the vertical domino rim-hooks from the first column and placing them horizontally to the first row, we again obtain the diagram $ (2n) $. Since this is clearly destructible and the two rules preserves destructibility, our initial diagram $ {\bm \lambda} $ is also destructible.

If $ m > 1 $, then, by Example \ref{Initial pattern}, the obtained diagram $ (2n-\frac{m(m+1)}{2},\bigtriangledown_{1} ^{m}) $ is indestructible. Since the first rule preserves indestructibility, our initial diagram $ {\bm \lambda} $ is also indestructible.

By reversing the above procedure, we can transform the diagram $ (2n) $ to any destructible one, and can transform the diagrams $ (2n-\frac{m(m+1)}{2},\bigtriangledown_{1} ^{m}) $, for suitable $ m > 1 $, to any indestructible one.
\end{proof}

\begin{corollary}
Let $ {\bm \lambda}$ and $ {\bm \gamma} $ be two destructible diagram with $ | {\bm \lambda} |=| {\bm \gamma}|$. Then diagrams $ {\bm \lambda} $ and $ {\bm \gamma} $ can be obtained from each other by applying the above two rules.
\end{corollary}

We remark that the above corollary does not always hold in indestructible case.
For example, it can be easily checked that two indestructible diagrams $ {\bm \lambda}=(6,3,2,1)=(6,\bigtriangledown_{1} ^{3}) $ and $ {\bm \gamma}=(9,2,1)=(9,\bigtriangledown_{1} ^{2}) $ with $ | {\bm \lambda} |=| {\bm \gamma}|=12 $ cannot be changed from one to another.
However, it may still hold in some cases. For example, the two indestructible diagrams $ {\bm \lambda}=(6,4,3,2,1)=(6,\bigtriangledown_{1} ^{4}) $ and $ {\bm \gamma}=(13,2,1)=(13,\bigtriangledown_{1} ^{2}) $ with $ | {\bm \lambda} |=| {\bm \gamma}|=16 $ can be changed from one to another as follows (first we transpose, and then we transfer domino rim-hooks to the first row):
$$ (6,\bigtriangledown_{1} ^{4}) \leftrightarrow (\bigtriangledown_{1} ^{5},1) \leftrightarrow (7,4,3,2) \leftrightarrow (9,4,3) \leftrightarrow (11,4,1) \leftrightarrow (13,2,1). $$


\section{\textbf{Characters induced by (in)destructible diagrams}}
\noindent
In this section we prove that the characters induced by indestructible diagram vanish identically on $ P_{2n} \subseteq S_{2n} $, and the characters induced by destructible diagram are always nonzero on $ P_{2n} $ (recall that $ P_{2n} $ denotes the set of permutations from $ S_{2n} $ whose cycle type contains only cycles of even length).

Our main tool will be the recursive \textit{Murnaghan-Nakayama rule} (see \cite[Theorem 4.10.2]{Sa} and {\cite[Corollary 7.2]{St}}, or \cite[Problem 4.45]{Ful}) which is a combinatorial way of computing the value of the irreducible characters $ \chi^{{\bm \lambda}} $. The crucial objects that come into play are the rim-hooks. The \textit{length} of a rim-hook $ {\bm {\bm \zeta}} $ is the number of cells in $ {\bm {\bm \zeta}} $. 
The \textit{leg length} of $ {\bm {\bm \zeta}} $ is $ ll({\bm {\bm \zeta}}) $=(the number of rows of $ {\bm {\bm \zeta}} $) - 1. 
Let us state Murnaghan-Nakayama rule for convenience.
\begin{theorem}\textnormal{(\cite[Theorem 4.10.2]{Sa} and {\cite[Corollary 7.2]{St}})}\label{Murnaghan-Nakayama Rule}
If $ {\bm \lambda} $ is a diagram with $ |{\bm \lambda}|=n $ and $ \rho=\rho_{1}\dots\rho_{k} $ is a cycle decomposition of a permutation $ \rho \in S_{n} $ into disjoint cycles of decreasing lengths, then we have
\begin{equation*}\label{M.N.R}
\chi^{{\bm \lambda}}({\rho})=\sum_{{\bm {\bm \zeta}}}(-1)^{ll({\bm {\bm \zeta}})}\chi^{{\bm \lambda} \backslash {\bm {\bm \zeta}}}({\rho \backslash \rho_{1}}),
\end{equation*}
where the sum runs over all rim-hooks $ {\bm {\bm \zeta}} $ of $ {\bm \lambda} $ with $ |\rho_{1}| $ cells. If $ {\bm \lambda} $ contains no rim-hooks of length $ |\rho_{1}| $, then $ \chi^{{\bm \lambda}}({\rho})=0 $.
\end{theorem}

We will also rely on the following result.
\begin{lemma}\textnormal{\cite[Lemma 7.3]{St}}\label{St2}
Let $ {\bm \lambda} $ be a diagram, and let $ {\bm {\bm \zeta}}={\bm \lambda} \setminus {\bm \gamma} $ be a rim-hook of $ {\bm \lambda} $ with $ |{\bm {\bm \zeta}}|=pk $, where $ p $ and $ k $ are integers. Then there exists a recursive procedure 
$$ \{ {\bm \lambda}={\bm \lambda}^{(0)} \rightarrow {\bm \lambda}^{(1)} \rightarrow \dots \rightarrow {\bm \lambda}^{(k)}={\bm \gamma} \} $$
such that for any $ i \in \{1, \dots, k \} $, $ {\bm \lambda}^{(i-1)} \setminus {\bm \lambda}^{(i)} $ is a rim-hook of $ {\bm \lambda}^{(i-1)} $ with $ |{\bm \lambda}^{(i-1)} \setminus {\bm \lambda}^{(i)}|=p $.
\end{lemma}

\begin{theorem}\label{Main1}
The character induced by indestructible diagram $ {\bm \lambda} $ with $ | {\bm \lambda} |=2n $ vanishes identically on $ P_{2n} $.
\end{theorem}
\begin{proof}
Let 
$$ \rho=\rho_{1} \dots \rho_{k} $$
be a cycle decomposition of permutation $ \rho \in P_{2n} $ into disjoint cycles whose orders are even and decreasing. By Murnaghan-Nakayama rule $ \chi^{{\bm \lambda}}({\rho})=\sum_{{\bm {\bm \zeta}}_{1}}(-1)^{ll({\bm {\bm \zeta}}_{1})}\chi^{{\bm \lambda} \backslash {\bm {\bm \zeta}}_{1}}({\rho \backslash \rho_{1}})$, where the sum runs over all rim-hooks $ {\bm {\bm \zeta}}_{1} $ of $ {\bm \lambda} $ with $ |\rho_{1}| $ cells. We will decompose this sum into two parts
\begin{equation*}
\chi^{{\bm \lambda}}(\rho)=\sideset{}{'}\sum_{{\bm {\bm \zeta}}_{1}}(-1)^{ll({\bm {\bm \zeta}_{1}})}\chi^{{\bm \lambda} \backslash {\bm {\bm \zeta}}_{1}}({\rho \backslash \rho_{1}}) + \sideset{}{''}\sum_{{\bm {\bm \zeta}}_{1}}(-1)^{ll({\bm {\bm \zeta}}_{1})}\chi^{{\bm \lambda} \backslash {\bm {\bm \zeta}}_{1}}({\rho \backslash \rho_{1}}),
\end{equation*}
where in $ \sum^{'} $ we sum up over all the rim-hooks $ {\bm {\bm {\bm \zeta}}}_{1} \subseteq {\bm \lambda} $ of length $ |\rho_{1}| $ such that the diagram $ {\bm \lambda} \setminus {\bm {\bm \zeta}}_{1} $ contains at least one rim-hook of length $ |\rho_{2}| $, while in $ \sum^{''} $ we sum up over all the rim-hooks of length $ |\rho_{1}| $ without this property. Every summand $ \chi^{{\bm \lambda} \backslash {\bm {\bm \zeta}}_{1}}({\rho \backslash \rho_{1}}) $ in $ \sum^{''} $ vanishes, which is seen by computing it again by Murnaghan-Nakayama rule. Therefore $ \sum^{''} $ vanishes and hence
$
\chi^{{\bm \lambda}}(\rho)=\sideset{}{'}\sum_{{\bm {\bm \zeta}}_{1}}(-1)^{ll({\bm {\bm \zeta}}_{1})}\chi^{{\bm \lambda} \backslash {\bm {\bm \zeta}}_{1}}({\rho \backslash \rho_{1}}).
$

Now, we can recursively continue the process on each summand $\chi^{{\bm \lambda} \backslash {\bm {\bm \zeta}}_{1}}({\rho \backslash \rho_{1}}) $ from $ \sum^{'} $. After $ r $ steps we get
\begin{equation*}
\chi^{{\bm \lambda}}(\rho)=\sum_{{\bm {\bm \zeta}}_{1}, \dots,{\bm {\bm \zeta}}_{r}}(-1)^{ll({\bm {\bm \zeta}}_{1})+ \dots+ll({\bm {\bm \zeta}}_{r})}\chi^{{\bm \lambda} \setminus \bigcup_{i=1}^{r}{\bm {\bm \zeta}}_{i}}(\rho_{r+1}\dots\rho_{k}).
\end{equation*}
where we sum up over all rim-hooks $ {\bm {\bm \zeta}}_{1}, \dots, {\bm {\bm \zeta}}_{r} $ in which $ {\bm {\bm \zeta}}_{i} $ is a rim-hook of $ {\bm \lambda} \setminus \bigcup_{j=1}^{i-1}{\bm {\bm \zeta}}_{j} $ with $| {\bm {\bm \zeta}}_{i} |= |\rho_{i}| $, $ i \in \{1, \dots, r \} $.
If at some step $ r < k $, $ \sum^{'} $ contains no summand, then $ \chi^{{\bm \lambda}}(\rho)=0 $. Otherwise, assume at each recursive step $ \sum^{'} $ contains at least one summand. This means that there is at least one way to recursively remove all the cells from $ {\bm \lambda} $ by removing first a rim-hook $ {\bm {\bm \zeta}}_{1} \subseteq {\bm \lambda} $ with $ | {\bm {\bm \zeta}}_{1} |= |\rho_{1}| $, then removing a rim-hook $ {\bm {\bm \zeta}}_{2} \subseteq {\bm \lambda} $ with $ | {\bm {\bm \zeta}}_{2} |= |\rho_{2}| $, etc. Since $ |\rho_{1}| $ is even, then by Lemma \ref{St2} the first step can be done on $ \dfrac{|\rho_{1}|}{2} $ sub-steps, where at each sub-step we recursively remove a domino rim-hook and always get a valid shape. Since $ |\rho_{2}| $ is even, similar conclusion holds for the second step, etc. This means that $ {\bm \lambda} $ is destructible, a contradiction to the original hypothesis.
\end{proof}

Let us now focus on irreducible characters of $ S_{2n} $ that correspond to destructible diagrams (see Theorem \ref{character} below).
Recall from Remark \ref{reversing} that every destructible diagram can be recursively built up by adding vertical/horizontal dominoes to its current border. It follows that every destructible diagram can be covered (perhaps in more than one way) with horizontal and vertical dominoes. For example, the diagram $ (2,2) $ has two different tilings $ D^{v}D^{v} $ and $ D^{h}D^{h} $, where $ D^{v} $ and $ D^{h} $ are vertical and horizontal dominoes, respectively.

To be more precise, a (domino) \textit{tiling} $ T $ of a diagram $ {\bm \lambda} $ is a placement of dominoes that covers all the cells of the diagram perfectly (i.e., no overlaps, no diagonal placements, no protrusions off the diagram).
We mention in passing that, alternatively, one can define a tiling to be a partition of the cells of $ {\bm \lambda} $ into a disjoint union of pairs of cells which share common edge (c.f. \eqref{DvDh}).

\begin{lemma}\label{pattern} 
Every tiling of a destructible diagram contains a domino rim-hook.
\end{lemma}
\begin{proof}
Let $ {\bm \lambda}=(\lambda_{1},\dots,\lambda_{t}) $ be a destructible diagram and $ T $ be one of its tilings. As in the proof of Lemma \ref{Triangle-Remove}, we denote the cells of row $ i $ in $ {\bm \lambda} $ by $ C_{i1}, C_{i2},\dots, C_{i\lambda_{i}} $, so each vertical/horizontal domino can be written as \eqref{DvDh}.

Since $ {\bm \lambda} $ is destructible, by Lemma \ref{Triangle-Remove}, it cannot be triangular. So, the difference between two consecutive rows is not always $ 1 $, here for simplicity, we consider $ {\bm \lambda} $ appended with a virtual (empty) last row containing no cells.

Assume first there exists a row $ k $ such that the difference between two consecutive rows, 
\begin{equation}\label{dk}
d_{k}=\lambda_{k}-\lambda_{k+1} \geq 2.
\end{equation}
Without loss of generality, we may assume that $ k $ is the smallest one. In this case, the two cells $ D=\{ C_{k(\lambda_{k}-1)}, C_{k\lambda_{k}} \} $ form a horizontal domino rim-hook. If $ D $ is a part of the tiling $ T $ we are done. Otherwise, we must have $ k \geq2 $ and the cell $ C_{k\lambda_{k}} $ must be covered with a vertical domino 
$$ D^{v}=\{ C_{(k-1)\lambda_{k}}, C_{k\lambda_{k}} \} \in T. $$ 

Now, if $ \lambda_{k-1}=\lambda_{k} $, since by \eqref{dk} we have $\lambda_{k} \ge \lambda_{k+1}+2$, then $D^{v} $ is a domino rim-hook, and again we are done. The only possibility left to consider is that 
$$ d_{k-1}=\lambda_{k-1}-\lambda_{k}=1. $$ 
In this case $ C_{(k-1)\lambda_{k-1}} $, the last cell of row $ k-1 $, cannot be covered with a horizontal domino $ D^{h} $, because $ D^{h} $ intersects the vertical domino $ D^{v} $ in the cell $ C_{(k-1)(\lambda_{k-1}-1)}=C_{(k-1)\lambda_{k}} $. So $ C_{(k-1)\lambda_{k-1}} $ is covered with a vertical domino 
$$ D_{1}^{v}=\{ C_{(k-2)\lambda_{k-1}}, C_{(k-1)\lambda_{k-1}} \} \in T. $$
Recall that $ k $ was the first row satisfying \eqref{dk}, and consequently $ d_{i}=\lambda_{i} - \lambda_{i+1} \in \{ 0, 1\} $ for every $ i < k-1 $. Therefore we can recursively proceed backwards to show that the last cell of each preceding row can only be covered with a vertical domino until either we hit a row with $ d_{i}=0 $, or else $ d_{i}=1 $ for $ 1\le i\le k-2 $ and we reach a first row. In the former case we stopped at row $ i+1 $ and cover a cell $ C_{(i+1)(\lambda_{i+1}-1)} $ with a vertical domino, and hence its edge-adjacent cell $ C_{(i+1)\lambda_{i+1}} $ can only be covered with a vertical domino $ D^{v}=\{ C_{i\lambda_{i+1}},C_{(i+1)\lambda_{i+1}}\}=\{ C_{i\lambda_{i}},C_{(i+1)\lambda_{i}}\} $. Since $ \lambda_{i}=\lambda_{i+1} = \lambda_{i+2} +1 $, then $ D^{v} $ is actually a domino rim-hook, and we are done once again. In the latter case, the cell $ C_{1\lambda_{1}} $ (the last cell of the first row) would also have to be covered with a vertical domino, a contradiction. Therefore, the latter case is impossible.

Assume the last $ d_{i} \leq 1 $ for every $ i $, and there exists $ k' < t $ with 
$$ d_{k'}=\lambda_{k'}-\lambda_{k'+1}=0. $$
We may suppose that $ k' $ is the largest one with this property. Notice that now $ \lambda_{t}=1 $ and $ d_{i}=1 $ for every $ i > k' $; that is, from row $ k'+1 $ on the diagram is of triangular form. If $ k'=t-1 $, then $ \lambda_{t-1}=\lambda_{t}=1 $; that is, each of the last two rows contains a single cell which form a vertical domino rim-hook $ D^{v} $. Clearly, the only cell of the last row cannot be covered with a horizontal domino, and therefore $ D^{v} \in T $. So, again we are done.

Finally, let $ k' < t-1 $. Since $ \lambda_{k'}=\lambda_{k'+1} >\lambda_{k'+2} $, then the last two cells of rows $ k' $ and $ k'+1 $ form a vertical domino rim-hook. If it is a part of the tiling $ T $, we are done. Otherwise, the last cell of row $ k'+1 $ must be covered with a horizontal domino 
$$ D^{h}=\{ C_{(k'+1)(\lambda_{k'+1}-1)}, C_{(k'+1)\lambda_{k'+1}} \} \in T. $$ 
Therefore, $ C $, the last cell of row $ k'+2 $, cannot be covered with a vertical domino, because it intersects $ D^{h} $ in the cell $ C_{(k'+1)(\lambda_{k'+1}-1)} $. So, also $ C $ is covered with a horizontal domino 
$$ D_{1}^{h}=\{ C_{(k'+2)(\lambda_{k'+2}-1)}, C \} \in T. $$
This allows us to recursively proceed forwards and show that the last cell of each succeeding row can only be covered with a horizontal domino. This is a contradiction, because the last row contains only one cell and it cannot be covered with a horizontal domino.
\end{proof}

The \textit{parity} is a mapping $ \pi\colon \mathbb{Z} \rightarrow \{ -1, 1 \}$ defined by $ \pi(k)=(-1)^{k} $. 
We also introduce the \textit{parity} of a diagram $ {\bm \lambda}=(\lambda_{1}, \dots, \lambda_{t}) $ by 
\begin{equation*}
\pi({\bm \lambda})=\prod_{i=1}^{t-1}\pi(\lambda_{i})^{t-i}.
\end{equation*}

\begin{lemma}\label{parity2}
Let $ {\bm \lambda}=(\lambda_{1}, \dots, \lambda_{t}) $ be a destructible diagram covered with dominoes. Then the parity of $ v $, the number of vertically placed dominoes, satisfies
\begin{equation*}
\pi(v)=\pi({\bm \lambda}).
\end{equation*}

\end{lemma}
\begin{proof}
For $ i \in \{1,\dots,t-1 \} $, let $ z_{i} $ be the number of vertically placed dominoes which intersect both rows $ i $ and $ i+1 $. Let also $ m_{i} $ be the number of vertically placed dominoes that intersect row $ i $ with one among its two cells, $ i \in \{1,\dots,t \}$. Clearly, 
\begin{equation}\label{1st}
m_{1}=z_{1};
\end{equation}
and moreover,
\begin{equation}\label{2nd}
m_{2}=z_{1}+z_{2}, \quad m_{3}=z_{2}+z_{3}, \quad \dots , \quad m_{t-1}=z_{t-2}+z_{t-1}, \quad \text{and} \quad m_{t}=z_{t-1}.
\end{equation}
Next notice that $ \lambda_{i}-m_{i} $ equals the number of horizontally placed dominoes in row
$ i $, so it must always be an even integer. Therefore, $1= \pi(\lambda_{i}-m_{i})=\pi(\lambda_{i})/\pi(m_{i}) $, and so
\begin{equation}\label{3rd}
\pi(\lambda_{i})=\pi(m_{i}).
\end{equation}
From \eqref{3rd} and \eqref{1st} we have $ \pi(\lambda_{1})=\pi(m_{1})=\pi(z_{1}) $. Moreover, from \eqref{3rd} and \eqref{2nd} we have $ \pi(\lambda_{2})=\pi(m_{2})=\pi(z_{1}+z_{2})=\pi(z_{1})\pi(z_{2})=\pi(\lambda_{1})\pi(z_{2}) $. Therefore, $ \pi(z_{2})=\pi(\lambda_{1})\pi(\lambda_{2}) $. Also, we have $\pi(\lambda_{3})=\pi(m_{3})=\pi(z_{2}+z_{3})=\pi(z_{2})\pi(z_{3})=\pi(\lambda_{1})\pi(\lambda_{2})\pi(z_{3}) $. Then $ \pi(z_{3})=\pi(\lambda_{1})\pi(\lambda_{2})\pi(\lambda_{3}) $. Proceeding recursively for any $ i \in \{1,\dots,t-1 \} $ we have $ \pi(z_{i})=\pi(\lambda_{1})\pi(\lambda_{2}) \dots \pi(\lambda_{i}) $. Hence,
\begin{align*}
\pi(v)&=\pi(z_{1}+ \dots + z_{t-1})\\
&=\pi(z_{1})\pi(z_{2}) \dots \pi(z_{t-1})\\
&=\pi(\lambda_{1}) \times \pi(\lambda_{1})\pi(\lambda_{2}) \times \dots \times
\pi(\lambda_{1})\pi(\lambda_{2}) \dots \pi(\lambda_{t-1})\\
&=\pi(\lambda_{1})^{t-1}\times \pi(\lambda_{2})^{t-2} \times \dots \times \pi(\lambda_{t-1})\\
&=\prod_{i=1}^{t-1}\pi(\lambda_{i})^{t-i}=\pi({\bm \lambda}).
\end{align*}
\end{proof}

\begin{remark}\label{vertical}
\textnormal{
Let a destructible diagram $ {\bm \lambda}=(\lambda_{1}, \dots, \lambda_{t}) $ be covered with dominoes and let $ m_{i} $ be as above. Then $ v $ and $ h $, the numbers of vertically/horizontally placed dominoes in any tiling of $ {\bm \lambda} $, satisfy 
\begin{equation*}
v=\frac{m_{1}+\dots + m_{t}}{2}=\frac{|{\bm \lambda} |-h}{2}.
\end{equation*}
Namely, we have $ m_{1}=z_{1} $, $ m_{2}=z_{1}+z_{2} $, $ \dots $, $ m_{t-1}=z_{t-2}+z_{t-1} $ and $ m_{t}=z_{t-1} $. 
Since $ v=z_{1}+\dots + z_{t-1} $, then by summing up we have $ m_{1}+\dots + m_{t}=2(z_{1}+\dots + z_{t-1})=2v $.
The second equation follows from $ 2(v+h)=|{\bm \lambda} | $, because each domino occupies two cells.
}
\end{remark}

Given a tiling $ T $ of a destructible diagram ${ \bm \lambda } $, the number of different ways to recursively remove all the dominoes from $ { \bm \lambda } $, where at each step we remove a single domino rim-hook by following the tiling $ T$, will be denoted by $ |T| $ (see Fig. \ref{decomposition3333}).

\begin{theorem}\label{character}
Let $ {\bm \lambda} $ be a destructible diagram of $ | {\bm \lambda} |=2n $, and let $ \mathcal{T} $ be the set of all possible tilings of $ \lambda $. If the cycle type of a permutation $ \rho \in S_{2n} $ is $ (2, \dots, 2) $, then
\begin{equation*}
\chi^{{\bm \lambda}}(\rho)=\pi({\bm \lambda})\sum_{T \in \mathcal{T}} |T| \neq 0.
\end{equation*}
\end{theorem}
\begin{proof}
Let $ \rho=\rho_{1}\dots \rho_{n} $ be a decomposition of $ \rho $ into disjoint cycles, each of order two.
By Murnaghan-Nakayama rule, $ \chi^{{\bm \lambda}}(\rho)=\sum_{D_{1}}(-1)^{ll(D_{1})}\chi^{{\bm \lambda} \backslash D_{1}}(\rho \setminus \rho_{1})$, where the sum runs over all domino rim-hooks $ D_{1} $ of $ {\bm \lambda} $. By Lemma \ref{dest. domino}, the diagram $ {\bm \lambda} \setminus D_{1} $ is still destructible, so it again contains at least one domino rim-hook.

This allows us to recursively continue the Murnaghan-Nakayama rule on each summand $\chi^{{\bm \lambda} \backslash D_{1}}(\rho \setminus \rho_{1}) $. After $ r $ steps we get
\begin{equation*}
\chi^{{\bm \lambda}}(\rho)=\sum_{D_{1}, \dots,D_{r}}(-1)^{ll(D_{1})+ \dots+ll(D_{r})}\chi^{{\bm \lambda} \setminus \bigcup_{i=1}^{r}D_{i}}(\rho_{r+1} \dots \rho_{n}),
\end{equation*}
where the summation is over all dominoes $ D_{1}, \dots, D_{r} $ such that for each $ i \in \{1, \dots, r \} $, domino $ D_{i} $ is a domino rim-hook of $ {\bm \lambda} \setminus \bigcup_{j=1}^{i-1}D_{j} $.
Since $ {\bm \lambda} $ is destructible, we will remove all the cells from $ {\bm \lambda} $ after $ n $ steps and get
\begin{equation}\label{M.N.R.2}
\chi^{{\bm \lambda}}(\rho)=\sum_{D_{1}, \dots,D_{n}}(-1)^{ll(D_{1})+ \dots+ll(D_{n})}.
\end{equation}
Now, we will rearrange this sum. Each summand corresponds to a tiling $ T $ of $ { \bm \lambda } $ determined by vertical/horizontal dominoes $ D_{1}, \dots, D_{n} $. Moreover, the exponent ${ll(D_{1})+ \dots+ll(D_{n})} $ coincides with the number of vertically placed dominoes in tiling $ T $. So by Lemma \ref{parity2} the value of this summand, $ (-1)^{ll(D_{1})+ \dots+ll(D_{n})} $ is equal to $\pi({\bm \lambda})$, the parity of $ { \bm \lambda } $.

Notice that more than one summand may correspond to the same tiling $ T $. This happens when during the recursive Murnaghan-Nakayama procedure we encounter a step where we have more than one possibility to remove domino rim-hook by following the given tiling. It follows that in \eqref{M.N.R.2} there are exactly $ |T| $ summands corresponding to the tiling $ T $.

Conversely, given any tiling $ T $ of $ { \bm \lambda } $, it contains, by Lemma \ref{pattern}, at least one domino rim-hook. This allows us to recursively remove all the dominoes from $ T $ in such a way that at each step one of the current domino rim-hooks is removed. If we ignore the sign, this corresponds exactly to one of the summands of \eqref{M.N.R.2} obtained after finishing the recursive Murnaghan-Nakayama procedure.
This shows that the equation \eqref{M.N.R.2} equals to $ \pi({\bm \lambda}) \sum |T|$, where we sum up over all the tilings $T$ of ${\bm \lambda}$, as claimed.
\end{proof}

Recall that $ P_{2n} $ is the subset of all permutations in the symmetric group $ S_{2n} $ with no cycles of odd length in the decomposition into the product of disjoint cycles; that is, the set of all permutations of even length. Now, we have the following corollary:
\begin{corollary}\label{irr-ch}
The following are equivalent for the irreducible character $ \chi $ of $ S_{2n} $:
\begin{itemize}
\item[(i)] $ \chi $ is induced by the indestructible diagram.
\item[(ii)] $ \chi(\rho)=0 $ for a permutation $ \rho=(1,2)(3,4) \dots(2n-1,2n) $.
\item[(iii)] $ \chi $ vanishes identically on the subset $ P_{2n} \subseteq S_{2n} $.
\end{itemize}
\end{corollary}
\begin{proof}
\mbox{}\\
(i) $ \Rightarrow $ (iii) has been proven in Theorem \ref{Main1}.\\
(iii) $ \Rightarrow $ (ii) is obvious.\\
(ii) $ \Rightarrow $ (i) follows from Theorem \ref{character}.
\end{proof}

We conclude the paper by classifying the immanants which vanish identically on alternate matrices $ \mathbb{A}_n(\CC) $.

\begin{corollary}\label{m.cor}
Let $d_{\chi}\colon \mathbb{M}_n(\CC)\to \CC$ be an immanant induced by an irreducible
character $\chi$ of $S_n$. Then the following are equivalent.
\begin{itemize}
\item[(i)] $d_\chi$ vanishes identically on a subspace $  \mathbb{A}_n(\CC) $ of alternate matrices.
\item[(ii)] Either $n$ is odd, or $n$ is even and $d_\chi(J\oplus\dots \oplus J)=0$ where $J=\left(\begin{smallmatrix}
0&1\\
-1&0\end{smallmatrix}\right) $.
\item[(iii)] Either $n$ is odd, or $n$ is even and $\chi$ is induced by an indestructible diagram.
\end{itemize}
\end{corollary}
\begin{proof}
\mbox{}\\
(iii) $ \Rightarrow $ (i). Let $ A $ be an alternate matrix in $  \mathbb{M}_n(\CC) $, and $ A^{T} $ be its transpose. If $ n $ is odd, then 
$d_{\chi}(A)=0$.
Now, let $ n $ be even and $\chi$ be induced by an indestructible diagram. Then, by Theorem \ref{Main1}, $ \chi $ vanishes identically on every permutation $ \rho \in P_{n} $. Therefore, by Proposition \ref{impn}, $d_\chi$ vanishes identically on every alternate matrix.\\
(i) $ \Rightarrow $ (ii) is trivial since $ J\oplus\dots \oplus J $ is an alternate matrix.\\
(ii) $ \Rightarrow $ (iii). Let $ n $ be even and let 
\begin{equation*} 
A=(a_{ij})=J\oplus\dots \oplus J=\begin{bmatrix}
\left(\begin{smallmatrix}
0 & 1\\
-1 & 0
\end{smallmatrix}\right) & 0 & \dots & 0\\

0 &
\left(\begin{smallmatrix}
0 & 1\\
-1 & 0
\end{smallmatrix}\right) & \dots & 0 \\

\vdots & \vdots & \ddots & \vdots \\

0 &
0 & \dots &
\left(\begin{smallmatrix}
0 & 1\\
-1 & 0
\end{smallmatrix}\right)
\end{bmatrix}.
\end{equation*}

Notice that if $ (i,j) \notin \{ (2k-1,2k), (2k, 2k-1) ; \; 1 \leq k \leq \frac{n}{2} \} $, then $ a_{ij}=0 $. So, in computing the immanant of $A$, every summand corresponding to a permutation $ \sigma$ with $ \sigma(2k-1)\neq2k $ or $ \sigma(2k)\neq 2k-1 $ for some $ k \in \{1, \dots, \frac{n}{2} \} $ is multiplied by $\prod_{i=1}^{n}a_{i\sigma(i)}=0$. Therefore
\begin{align*}
d_\chi(A)&=\chi((1,2)(3,4) \dots(n-1,n)) \prod_{i=1}^{\frac{n}{2}}a_{(2i-1)(2i)} \, a_{(2i)(2i-1)}\\
&=\chi((1,2)(3,4) \dots(n-1,n)) \cdot (-1)^{\frac{n}{2}}.
\end{align*}
Since $ d_{\chi}(A)=0 $, then $ \chi((1,2)(3,4) \dots(n-1,n))=0 $. Therefore, by Corollary \ref{irr-ch}, $\chi$ is induced by an indestructible diagram.
\end{proof}

Let us finish with a remark that Corollary \ref{irr-ch} is not true in other fields. For example in fields of characteristic $ 2 $, character induced by a destructible diagram $ {\bm \lambda}=(4,4) $ vanishes identically on $ P_{8} $.
We do not know what happens in fields of other characteristic.
\section{Addendum}\label{Addendum}
\noindent
In \cite{Kuz}, the following result was stated.

\begin{theorem}\textnormal{\cite[Theorem 1]{Kuz}}\label{thm:main}
Let $n\ge3$ be an integer, let $\FF$ be a field with $|\FF|\ge n+1$, and let $\chi,\chi'$ be two irreducible complex
characters of the symmetric group $S_n$. Suppose $\Phi \colon \mathbb{M}_{n}(\FF)\to \mathbb{M}_{n}(\FF)$ is any map with the property
\begin{equation}\label{eq:idenitiy}%
d_{\chi}(A+\lambda B)=d_{\chi'}(\Phi(A)+\lambda\Phi(B));\qquad (A,B\in
\mathbb{M}_{n}(\FF);\;\lambda\in\FF).
\end{equation}
Then $\Phi$ is automatically linear and bijective.
\end{theorem}

However, the arguments which prove this result work only under the additional assumption that for every
pair of integers $(i,j)\in\{1,\dots,n\}$ there exists a permutation $\sigma\in S_n$ with $\sigma(i)=j$ and with
$\chi(\sigma)\ne 0\in \FF$ (this fact is needed in the proof of \cite[Lemma 4]{Kuz}, see also Example \ref{ex1} below).

With this additional assumption on the character $ \chi $, all the arguments in \cite{Kuz} are then valid, and hence prove the following (which corrects Theorem \ref{thm:main}).

\begin{theorem}\label{thm:main2}
Let~$n\ge3$ be an integer, let $\FF$ be a field with $|\FF|\ge n+1$, and let $\chi,\chi'$ be two irreducible complex
characters of the symmetric group $S_n$ such that for every pair of integers $(i,j)\in\{1,\dots,n\}$ there
exists a permutation $\sigma\in S_n$ with $\sigma(i)=j$ and $\chi(\sigma)\ne 0\in \FF$. Suppose $\Phi \colon \mathbb{M}_{n}(\FF) \to \mathbb{M}_{n}(\FF)$ is any map with the property
\begin{equation*}
d_{\chi}(A+\lambda B)=d_{\chi'}(\Phi(A)+\lambda\Phi(B));\qquad (A,B\in
\mathbb{M}_{n}(\FF);\;\lambda\in\FF).
\end{equation*}
Then $\Phi$ is automatically linear and bijective.
\end{theorem}
In fact, as we show next, this additional assumption is automatically satisfied if $ \charac(\FF)\neq 2 $. To make the proof easier, we introduce an additional terminology.

A rim-hook is \textit{complete} if it is of maximal possible length. Such a rim-hook consists of a chain of consecutive edge-connected cells which start at the last cell of the first row in a diagram and continue all the time leftwards or downwards ending up at the last cell of the first column. Since the chain starts at the last cell of the first row and ends up at the last cell of the first column, we have moved $ s-1 $ times leftwards and $ r-1 $ times downwards, where $ r $ is the number of rows and $ s $ is the number of columns in our diagram. By adding also the starting cell of the chain, we see that the length of a complete rim-hook is $ (s-1)+(r-1)+1=s+r-1 $. Notice that the length of a complete rim-hook is $ 1 $ if and only if $ |{\bm \lambda}|=1 $.

\begin{lemma}\label{newnew}
Let $ n \geq 3 $. If $ \chi $ is an irreducible character of $ S_n $, then for every $(i,j)\in\{1,\dots,n\}$ there
exists a permutation $\sigma\in S_n$ such that $\sigma(i)=j$ and $ \chi(\sigma) \in \{-2, -1,  1, 2 \}$.
\end{lemma}
\begin{proof}
Let $ {\bm \lambda} $ be a diagram with $ |{\bm \lambda}|=n \geq 3 $ such that $ \chi=\chi^{{\bm \lambda}} $, and let 
$$ R_{{\bm \lambda}}= \{ {\bm \lambda}={\bm \lambda}^{(0)} \rightarrow {\bm \lambda}^{(1)} \rightarrow \dots \rightarrow {\bm \lambda}^{(j)}=(k,1^{t}) \} $$
be the recursive procedure starting with a diagram $ {\bm \lambda} $ whereby at step $ i $, a complete rim-hook is removed from the obtained diagram $ {\bm \lambda}^{(i-1)} $, $ i \in \{ 1, \dots, j \} $, until we reach a diagram $ {\bm \lambda}^{(j)} $ which is its own rim-hook; that is
$$ {\bm \lambda}^{(j)}=(k,1^{t}) $$
 for some $ 0 \leq k, t \leq n$. 

At each step $ i $, the length of the complete rim-hook of $ {\bm \lambda}^{(i)} $ equals $ \ell_{i}=s_{i}+r_{i}-1 $, where $ s_{i} $ and $ r_{i} $ denote the number of columns and rows of $ {\bm \lambda}^{(i)} $, respectively. Since for any $ i $, $ s_{i+1} < s_{i} $ and $  r_{i+1} < r_{i} $, then $ \ell_{i+1} < \ell_{i} $.

Let us now define the conjugacy classes of the desired permutations as follows:
\smallskip

\textbf{Case (i)}: $ j=0 $ and  $ {\bm \lambda}^{(j)}={\bm \lambda}=(2,1) $. 
From Murnaghan-Nakayama rule, one easily observes that 
$$ \chi^{(2,1)}(\mathrm{id})=2 \quad \quad \text{and} \quad \quad \chi^{(2,1)}(\sigma)=\chi^{(2,1)}(\sigma^{-1})=1, $$
where $ \sigma=(1,2,3) $ is a long cycle.
Notice that one among the permutations $ \mathrm{id}, \sigma , \sigma^{-1} $ maps $ i $ into $ j $.

In the remaining cases we will find a single permutation $ \sigma$ to do the job. In fact, it suffices to give its cycle type so that cycle lengths are not increasing and the first one has length at least two, while the last one has length one. If $ j > 0 $, we will start by cycle lengths $ \ell_{1}, \dots, \ell_{j} $. Since at each step $ i \in \{ 1, \dots, j \} $, there exists only one complete rim-hook (i.e., of length $ \ell_{i}$), then by Murnaghan-Nakayama rule (Theorem \ref{Murnaghan-Nakayama Rule}) any such permutation $ \sigma $ will satisfy
$ \chi^{{\bm \lambda}}(\sigma)=\pm\chi^{{\bm \lambda}}(\sigma \setminus \{\sigma_{1}, \dots, \sigma_{j} \}) $, where $ |\sigma_{i}|=\ell_{i} $, $ i \in \{ 1, \dots , j\}$.
\smallskip

\textbf{Case (ii)}: $ j \geq 0 $, and  $ {\bm \lambda}^{(j)}=(1) $ is a single cell. Since $ |{\bm \lambda}| \geq 3 $, then $ j > 0 $. The desired permutation $ \sigma $ is of cycle type $ (\ell_{1}, \dots, \ell_{j},1) $. Therefore $ \chi^{{\bm \lambda}}(\sigma)=\pm1$.
\smallskip

\textbf{Case (iii)}: $ j \geq 0 $, and $ k \neq t+1  $ with $ k \geq 2 $ and $ t \geq 1 $; i.e., $ (k,1^{t}) $ is a non-symmetric rim-hook with at least two columns and at least two rows. The desired permutation $ \sigma $ is of cycle type $ (\ell_{1}, \dots, \ell_{j}, M, 1^{m+1}) $, where $ M=\max\{ k-1, t \} $ and $ m=\min\{ k-1, t \} $. Therefore, by Murnaghan-Nakayama rule, $ \chi^{{\bm \lambda}}(\sigma)=\pm\chi^{(k,1^{t})}(\sigma_{j+1} \sigma_{j+2} \dots \sigma_{m+j+2} )=\pm\chi^{ \bm \alpha}(1^{m+1})=\pm1 $, where $ |\sigma_{j+1}|=M $, $ |\sigma_{j+2}|=\dots=|\sigma_{m+j+2}|=1 $, and the diagram $ {\bm \alpha} \in \{ (1^{m+1}), (m+1) \}  $.
\smallskip

\textbf{Case (iv)}: $ j > 0 $, and $ k \neq t+1  $ with either $ k=1 $ or $ t=0 $; i.e., ${\bm \lambda}^{(j)} $ is either a single column $ (1^{t+1}) $ or a single row $ (k) $. The desired permutation is of cycle type $ (\ell_{1}, \dots, \ell_{j-1}, 1^{M+1} )$, where $ M=\max\{ k-1, t \} $. Therefore, by Murnaghan-Nakayama rule,  $ \chi^{{\bm \lambda}}(\sigma)=\pm\chi^{{\bm \beta}}(\sigma_{j+1} \dots \sigma_{M+j+1})=\pm1 $, where $ |\sigma_{j+1}|=\dots=|\sigma_{M+j+1}|=1 $ and the diagram $ {\bm \beta} \in \{ (1^{M+1}), (M+1) \} $.
\smallskip

\textbf{Case (v)}: $ j = 0 $, and $ k \neq t+1  $ with either $ k=1 $ or $ t=0 $; i.e., ${\bm \lambda}={\bm \lambda}^{(j)} $ is either a single column $ (1^{t+1}) $ or a single row $ (k) $. The desired permutation $ \sigma $ is of cycle type $ (2, 1^{n-2}) $, and by Murnaghan-Nakayama rule, $ \chi^{{\bm \lambda}}(\sigma)=\pm 1$.
\smallskip

\textbf{Case (vi)}: $ j \geq 0 $, and $ k=t+1 > 2 $; i.e., $ {\bm \lambda}^{(j)}=(k,1^{k-1}) $ is a symmetric rim-hook. The desired permutation $ \sigma $ is of cycle type $ (\ell_{1}, \dots, \ell_{j},k-1, k-1, 1) $.
Now, there are two rim-hooks of lengths $ k-1 $ in $ {\bm \lambda}^{(j)} $:
the horizontal one $ R^{h} $ and the vertical one $ R^{v} $. If at step $ j+1 $, we remove $ R^{v} $ (respectively, $ R^{h} $), then the obtained diagram is a single row $ {\bm \lambda}^{(j+1)}=(k) $ (respectively, a single column $ {\bm \lambda}^{(j+1)}=(1^{k}) $). So, at step $  j+2 $, we remove from $ (k)$ (respectively, $ (1^{k})$) the horizontal (respectively, vertical) rim-hook of length $ k-1 $. In both cases $ {\bm \lambda}^{(j+2)}=(1) $ is obtained. 
Therefore, by Murnaghan-Nakayama rule, $ \chi^{{\bm \lambda}}(\sigma)=\pm\chi^{(k,1^{k-1})}(\sigma_{j+1}\sigma_{j+2}\sigma_{j+3})$, where $ |\sigma_{j+1}|=|\sigma_{j+2}|=k-1 $ and $ |\sigma_{j+3}|=1 $.~Then
 \begin{align*}
 \chi^{(k,1^{k-1})}(\sigma_{j+1}\sigma_{j+2}\sigma_{j+3})&=(-1)^{(k-1)-1}\chi^{(k)}(\sigma_{j+2}\sigma_{j+3})+\chi^{(1^{k})}(\sigma_{j+2}\sigma_{j+3})\\
 &=(-1)^{k-2}\chi^{(1)}(\sigma_{j+3})+(-1)^{(k-1)-1}\chi^{(1)}(\sigma_{j+3})\\
 &=2(-1)^{k-2}\chi^{(1)}(\sigma_{j+3})=2(-1)^{k-2}.
 \end{align*}
 Therefore $ \chi^{{\bm \lambda}}(\sigma)=\pm2 $.
\end{proof}

\begin{corollary}
Let $n\ge3$ be an integer, let $\FF$ be a field with $|\FF|\ge n+1$, and  let $\chi,\chi'$ be two irreducible complex
characters of the symmetric group $S_n$. Suppose $\Phi \colon \mathbb{M}_n(\FF)\to \mathbb{M}_n(\FF)$ is any map with the property
\begin{equation*}
d_{\chi}(A+\lambda B)=d_{\chi'}(\Phi(A)+\lambda\Phi(B));\qquad (A,B\in
\mathbb{M}_{n}(\FF);\;\lambda\in\FF).
\end{equation*}
If $ \charac(\FF)\neq 2 $, then $\Phi$ is automatically linear and bijective.
\end{corollary}
\begin{proof}
Assume $ \charac(\FF)\neq 2 $. It follows from Lemma \ref{newnew} that for every $(i,j)\in\{1,\dots,n\}$ there
exists a permutation $\sigma\in S_n$ such that $\sigma(i)=j$ and $ \chi(\sigma) \neq 0\in\FF$. The rest follows from \cite{Kuz} (see also Theorem \ref{thm:main2}).
\end{proof}

Let us finish by showing that in the fields of characteristic $ 2 $, the results are different. 

\begin{example}\label{ex1}
\textnormal{
Assume $n=5$ and
let the irreducible character $\chi^{(3,1,1)} \colon S_5\to\ZZ$ correspond to a
partition $5=3+1+1$.  By Murnaghan-Nakayama rule, one can show that
$$\chi^{(3,1,1)} \colon \left(\begin{smallmatrix}
(1)&(1,2) &(1,2)(3,4)&(1,2,3)&(1,2,3)(4,5)&(1,2,3,4)&(1,2,3,4,5)\\
6&0 & -2 &0&0&0&1\end{smallmatrix}\right).$$
Consequently, if $\FF$ is a field of characteristic $ 2 $, then no
permutation $ \sigma $ with $\chi^{(3,1,1)}(\sigma)\ne0\pmod 2$ satisfies $\sigma(1)=1$. 
Moreover, the conclusions of Theorem \ref{thm:main} are not valid for $\chi=\chi'=\chi^{(3,1,1)}$ if $\FF$ is a field of characteristic $ 2 $. Namely, $d_{\chi}(A+\alpha E_{11})=d_{\chi}(A)$ implies that the non-linear and non-bijective map $\Phi\colon \mathbb{M}_5(\FF)\to \mathbb{M}_5(\FF)$, defined by $\Phi \colon X\mapsto X+E_{11}$, satisfies \eqref{eq:idenitiy}.
}
\end{example}

\newpage
\section{\textbf{Figures}}
\begin{figure}[h!]
\begin{center}
\includegraphics[width =9cm]{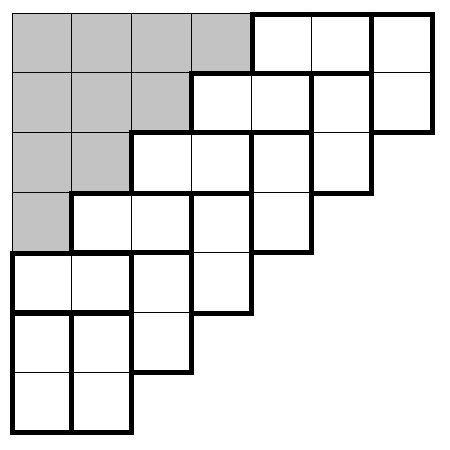}
\caption{An example of an indestructible diagram (see Lemma \ref{Triangle-Remove}).}
\end{center}
\end{figure}

\newpage
\begin{figure}[h!]
\begin{center}
\includegraphics[width =16cm]{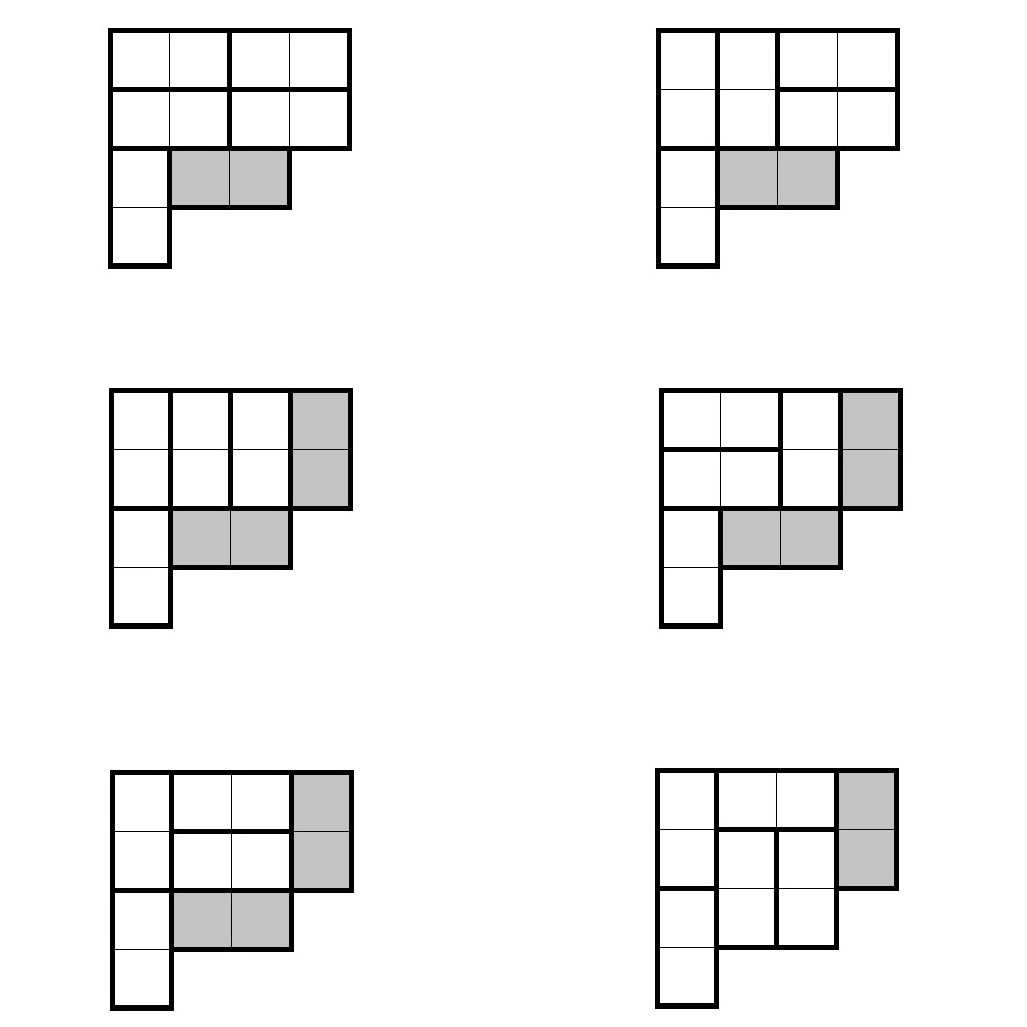}
\caption{All six different tilings of diagram $ (4,4,3,1) $. Notice that the number of vertical dominoes is always odd (c.f.~Lemma \ref{parity2}). Notice also that each tiling contains at least one (shaded) domino rim-hook (c.f.~Lemma \ref{pattern}).}
\end{center}
\end{figure}

\newpage
\begin{figure}[!h]
\centering{\includegraphics[width =15cm]{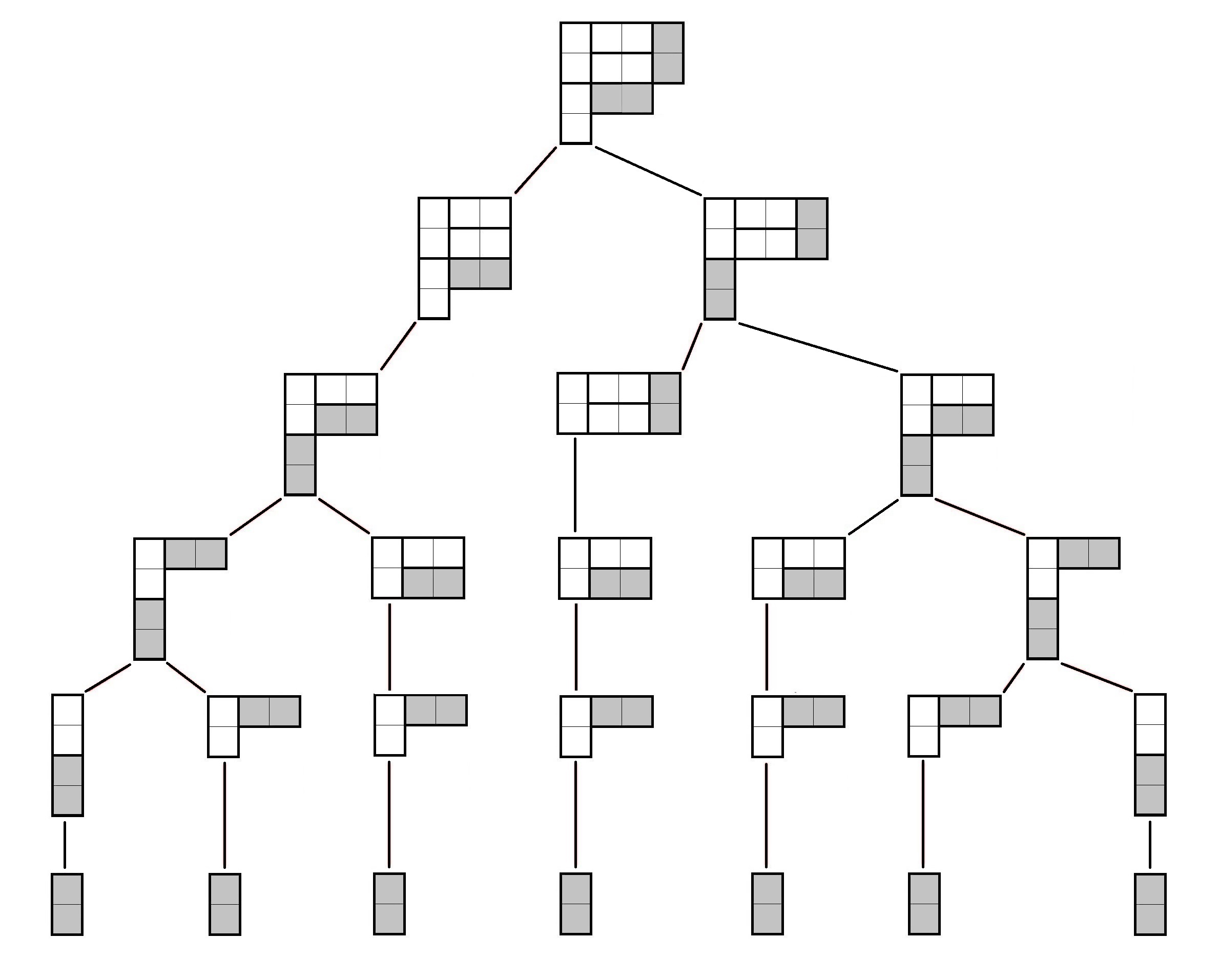}}
\caption{There are $|T|=7$ possibilities to recursively remove rim-hook dominoes by following a fixed tiling $ T $ of diagram $ (4,4,3,1) $ (c.f.~Theorem \ref{character} and the text immediately before it). The domino rim-hooks of each obtained diagram are shaded. \label{decomposition3333}}
\end{figure}

\end{document}